\documentclass[10pt]{amsart}
\usepackage{amssymb, tikz, mathrsfs}
\usepackage[all]{xy}
\usepackage{graphicx}

\title{Critical dynamics of variable-separated affine correspondences}
\author{Patrick Ingram}
\date{\today}
\address{Fort Collins, Colorado, United States of America}

\newcommand{\QQ}{\mathbb{Q}}
\newcommand{\ZZ}{\mathbb{Z}}
\newcommand{\CC}{\mathbb{C}}
\newcommand{\RR}{\mathbb{R}}
\newcommand{\FF}{\mathbb{F}}
\newcommand{\PP}{\mathbb{P}}
\renewcommand{\AA}{\mathbb{A}}

\newcommand{\Ocal}{\mathcal{O}}

\newcommand{\pf}{\mathfrak{p}}

\newcommand{\MOD}[1]{~(\textup{mod}~#1)}

\newcommand{\basin}{\lambda}
\newcommand{\path}{\mathscr{P}}

\renewcommand{\phi}{\varphi}
\renewcommand{\epsilon}{\varepsilon}

\newtheorem{theorem}{Theorem}[section]
\newtheorem{lemma}[theorem]{Lemma}
\newtheorem{proposition}[theorem]{Proposition}
\newtheorem{conjecture}[theorem]{Conjecture}

\theoremstyle{remark}
\newtheorem{remark}[theorem]{Remark}

\theoremstyle{definition}
\newtheorem{definition}[theorem]{Definition}

\begin{document}
\begin{abstract}
We examine affine correspondences of the form $g(y)=f(x)$, for $f$ and $g$ polynomials satisfying $\deg(g)<\deg(f)$, with the property that every critical point of the correspondence admits at least one finite forward orbit. In the case $g(y)=y$, this reduces to the study of post-critically finite polynomials, and our main result extends earlier finiteness results of the author. Specifically, we show that the collection of such correspondences of a given bidegree coincides with a subset of the parameter space of bounded Weil height. We also show that there are no non-trivial holomorphic families of correspondences with the above-described property.
\end{abstract}

\maketitle


\section{Introduction}

Working for the moment over $\CC$, a \emph{correspondence} on $\AA^1$ is a closed curve $C\subseteq\AA^2$ with the property that both coordinate projections $x, y:C\to\AA^1$ are finite and surjective, and a \emph{critical point} of such an object is a point $a\in \AA^1$ such that $C$ is not locally the graph of a biholomorphism at some point with $x=a$. We are interested in studying the critical dynamics of correspondences, that is, the behaviour of sequences $x_n\in \CC$ such that $x_0$ is a critical point of $C$, and such that $(x_n, x_{n+1})\in C$ for all $n$; such a sequence will be called a \emph{path} originating at $x_0$. In the case that $C$ is the graph of a polynomial, this corresponds to studying the forward orbits of critical points in the usual sense, which are well-known to reveal a great deal about the general dynamics of the function in question. In the case of correspondences, there is some indication that critical orbits are equally important~\cite{bullett2}.

Given the significance of critical orbits in the dynamics of rational functions, it is not surprising that the case of postcritically finite (PCF) morphisms has received much attention. Bullett~\cite{bullett1} has considered correspondences which are \emph{critically finite}, but it is worth noting that his definition (that all critical points have finite grand orbit) is very restrictive. Indeed, the graph of a PCF  polynomial is generally not a critically finite correspondence in this sense.

We will say that a correspondence is \emph{postcritically constrained} (PCC) if and only if each critical point admits at least one finite forward orbit, that is, one preperiodic path. It is easy to check that the correspondence $y=f(x)$ is PCC if and only $f$ is PCF. We will in particular consider the case of \emph{variable-separated} correspondences, i.e., those of the form
\begin{equation}\label{eq:varsep}C:g(y)=f(x),\end{equation}
for polynomials $f$, $g$ satisfying $\deg(g)<\deg(f)$.


Our first result is arithmetic in nature, showing that the set of PCC correspondences of a given bidegree is parametrized by a set of bounded height.

\begin{theorem}\label{th:heightbound}
Let $a_i, b_j\in \CC$ such that
\begin{equation}\label{eq:normal}y^e+b_{e-1}y^{e-1}+\cdots +b_1y = x^d+a_{d-1}x^{d-1}+\cdots + a_1x\end{equation}
is PCC, where we assume $d>e\geq 1$. Then the $a_i$ and $b_j$ are algebraic numbers of height bounded in terms of $d$ and $e$.
 In particular, there are only finitely many PCC correspondences in the form~\eqref{eq:normal} with coefficients of algebraic degree at most $D$, for any fixed $D$.
\end{theorem}
We note that, over an algebraically closed field, any correspondence of the form~\eqref{eq:varsep} can be put in the form~\eqref{eq:normal} by an affine-linear change of variables.

Theorem~\ref{th:heightbound} follows immediately from a stronger result. We define a natural height $h_{\mathrm{Weil}}(C)$ on a correspondence $C$ of the above form, derived from the usual Weil height of the tuple of coefficients in a certain weighted projective space. We also define a \emph{critical height} $h_{\mathrm{Crit}}(C)$ which  vanishes on PCC correspondences. It is not \emph{a priori} obvious that this critical `height' is related in any way to a Weil height on the appropriate variety but, as it betides, the two functions  are essentially the same.
\begin{theorem}\label{th:main}
For any correspondence $C$ as above we have
\[h_{\mathrm{Crit}}(C)=h_{\mathrm{Weil}}(C)+O(1),\]
where the implied constants depend only on $d$ and $e$, and can be made explicit.
\end{theorem}
Theorem~\ref{th:main} is obtained by decomposing both sides into sums of local heights, and proving an analogous inequality at each place. The local inequalities appear in Section~\ref{sec:greens}, and might be of independent interest for applications in the complex or $p$-adic dynamics of correspondences.

The critical height is defined more precisely in Section~\ref{sec:global}, and reduces to that used in~\cite{pcfpn} when $C$ takes the form $y=f(x)$ (this is not quite the same critical height as used in \cite{pcfpoly, barbados}). To aid the reader's intuition, however, we note one of its basic properties. In \cite{corr} we define a canonical height $\hat{h}_C$ on paths associated to a correspondence $C$, which has the property that $\hat{h}_C(P)=0$ if the path $P$ is preperiodic. We will see below that for any assignment of paths $c\mapsto P_c$, where $P_c$ is a path with initial vertex $c$, we have
\[h_{\mathrm{Crit}}(C)\leq \sum_{c\in\operatorname{Crit}(C)} \hat{h}_C(P_c),\]
where $\operatorname{Crit}(C)$ denotes the set of critical points of $C$.
In particular, if all of these paths are preperiodic (a possibility just in case $C$ is PCC), then we must have $h_{\mathrm{Crit}}(C)=0$. It is not clear that the converse holds, and this would be an interesting problem for further study.

Just as the results in \cite{pcfpoly, pcfpn} have applications in the geometric context, the arguments establishing Theorem~\ref{th:heightbound} may be applied over function fields with geometric consequences. 
Given a quasi-projective variety $X/k$, a \emph{family of (variable-separated affine) correspondences} $C/X$ is given by a pair of polynomials $f(z), g(z)\in k[X][z]$. Such a family is said to be \emph{split} if we may take the coefficients of $f$ and $g$ to be constant, after a change of variables, or \emph{isotrivial} if there exists a cover $Y\to X$ such that the base extension of this family to $Y$ is split. For instance, the family
\[y^3+\alpha^{-1}y=\alpha x^5+\alpha^{-1} x\]
 is split if $\alpha\in k[X]$ is a square, and hence its isotriviality is witnessed by some double cover of $X$.
\begin{theorem}\label{th:thurston}
Let $d>e$, let $k$ be an algebraically closed field of characteristic $0$ or $p>d$, let $X/k$ be a quasi-projective variety, and let $C/X$ be a family of correspondences of the above form all of which are PCC. Then the family $C/X$ is isotrivial.
\end{theorem}

Theorem~\ref{th:thurston} implies, for instance, that there are no non-trivial algebraic families of PCC correspondences of the form~\eqref{eq:varsep}, a fact which reduces in the case $\deg(g)=1$ to the non-existence of non-trivial families of postcritically finite polynomials.

Moving from correspondences on $\AA^1$ to those on $\PP^1$, we see a different picture. Give two \emph{rational} functions $f(x)$ and $g(y)$, we will call the correspondence $g(y)=f(x)$ on $\PP^1$ a \emph{Latt\`{e}s example} if and only if there exists an elliptic curve $E$, a surjective morphism $\pi:E\to\PP^1$, and two morphisms $\phi, \psi:E\to E$ making the following diagram commute:
\[\xymatrix{
E \ar[d]_{\pi} \ar[r]^{\phi}  & E \ar[d]_{\pi} & E \ar[d]^{\pi} \ar[l]_\psi\\  
\PP^1 \ar[r]_{f} & \PP^1 & \PP^1 \ar[l]^{g}
}\]
It is easy to construct non-trivial holomorphic families of PCC correspondences amongst the Latt\`{e}s examples, for instance taking $\pi$ to be the usual map to the Kummer surface, and $\phi, \psi\in \operatorname{End}(E)$. The following conjecture generalizes a consequence of a well-known theorem of Thurston~\cite{thurston}.

\begin{conjecture}
Any non-isotrivial family of PCC variable-separated correspondences on $\PP^1$ is a family of Latt\`{e}s examples.
\end{conjecture}

Similarly, we posit an arithmetic analogue, generalizing a conjecture of Silverman~\cite[Conjecture~6.30, p.~101]{barbados} which was verified by Benedetto, the author, Jones, and Levy~\cite{pcfrat}. Variable-separated correspondences of bidegree $(d, e)$ on $\PP^1$ are described by pairs of rational functions $f(x), g(y)$, which in turn are parametrized by affine varieties $\operatorname{Hom}_d$ and $\operatorname{Hom}_e$~\cite[\S~1.4]{barbados}. The variety $\operatorname{Corr}_{d, e}=\operatorname{Hom}_d\times \operatorname{Hom}_e$ is naturally acted upon by $\operatorname{PGL}_2^2$ (viewed as the automorphism group of $\PP^1\times\PP^1$), with
\[(\psi, \phi)\cdot (f, g)=(\psi\circ f\circ \phi^{-1}, \psi\circ g\circ\phi^{-1})\]
(note that many elements have stabilizers of positive dimension).
Assuming that it exists, we denote the categorical quotient by $\mathcal{M}_{d, e}=\operatorname{Corr}_{d, e}/\operatorname{PGL}_2^2$ (in the case $e=1$, this construction is described by Silverman~\cite{barbados}, and is shown by Levy~\cite{alon} to result in a rational affine variety).

\begin{conjecture}
The set of points in $\mathcal{M}_{d, e}$ corresponding to postcritically constrained correspondences is contained in the union of the set of Latt\`{e}s examples and a set of bounded height (with respect to some ample class).
\end{conjecture}

At this point, the reader might be forgiven for wondering whether or not many of the results above are vacuous. One checks, for instance, that $y^2=x^3+1$ defines a PCC correspondence on $\AA^1$, but this is just one example. Given $d>e$, are there in fact infinitely many (pairwise non-isomorphic) PCC correspondences of the form~\eqref{eq:normal}? We give an answer in the affirmative at least when $d$ is prime. 

\begin{theorem}\label{th:notallobvious}
Let $p>e$, with $p$ prime. Then there exist infinitely many distinct values $c\in\CC$ such that the correspondence on $\AA^1$ defined by $y^e=x^p+c$ is post-critically constrained.
\end{theorem}

The critical points of $y^3=x^d+c$ turn out to be $x=0$ and the roots of $x^d+c$, and so the correspondence is PCC if and only if there is a preperiodic path beginning at 0.
Note that it is clear, for each $n$, that there is a $c\in \CC$ such that some critical path of $y^e=x^d+c$ returns to $0$ after $n$ steps. What is not \emph{a priori} clear is that these $c$ are distinct as $n$ increases.

What we prove is in fact much stronger than Theorem~\ref{th:notallobvious}. Specifically, we show that for sufficiently large $m$ there exists a value $c\in\CC$ such that the critical point $x=0$ of $y^e=x^p+c$ lies in a periodic path of length exactly $m$.  Indeed, we prove that this holds over the algebraically closed field $\overline{\FF_p}$, rather than $\CC$, although it may not be \emph{a priori} clear that this is stronger.

We lastly bring our attention to the complex dynamics of correspondences, and prove a result analogous to the compactness of the Mandelbrot set (i.e., the connectedness locus). Specifically, let $S_{d, e}\subseteq \CC^{d+e-2}$ denote the set of tuples of coefficients for which the correspondence defined in \eqref{eq:normal} has the property that each critical point admits a bounded path (a set which clearly contains all PCC examples). In Figure~\ref{fig:m}, we depict the analogue of the Mandelbrot set for the family of correspondences $y^2=x^3+c$ (a 
closed subset of $S_{3, 2}$ of one complex dimension).

\begin{figure}\label{fig:m}
\includegraphics[scale=0.33]{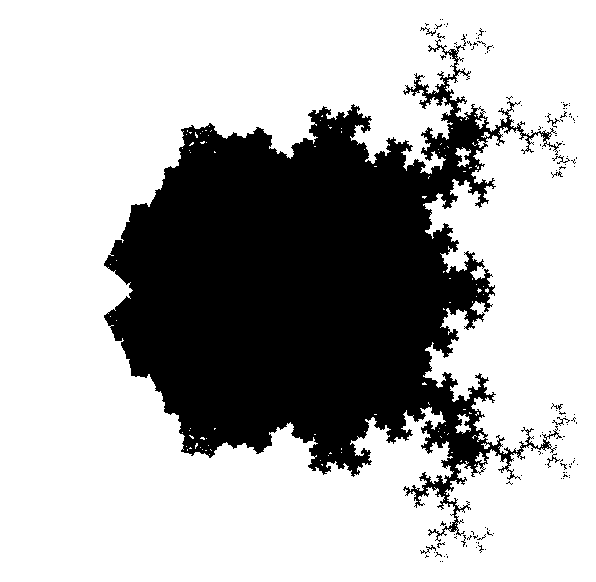}
\caption{Values $c\in\CC$ for which $y^2=x^3+c$ admits a bounded critical path}
\end{figure}

\begin{theorem}\label{th:mandelbrot}
The set $S_{d, e}$ is compact.
\end{theorem}
In Section~\ref{sec:greens} we describe an escape-rate function for paths which generalizes that normally used in polynomial dynamics. If we write $G_{\mathrm{min}}(a)$ for the minimum escape rate of any path starting at $a\in \CC$, and then $\Lambda(C)=\sum_{c\in \mathrm{Crit}(C)}G_{\mathrm{min}}(c)$, then $\Lambda(C)=0$ exactly on the set $S_{d, e}$. Note that, in the case $e=1$, $\Lambda$ essentially reduces to the Lyapunov exponent~\cite{mn, ma, prz} (see also~\cite{demarco}).

%


\section{Variable-separated affine correspondences}\label{sec:geom}

In this section we lay down the basics of variable-separated polynomial correspondences, working over an arbitrary algebraically closed field $k$. Let $g(z), f(z)\in k[z]$ be any polynomials. These polynomials define a \emph{correspondence} 
\[C=\{(a, b):g(b)=f(a)\}\subseteq \AA^2_k.\]
We will refer to the pair $(\deg(f), \deg(g))$ of integers as the \emph{bidegree} of the correspondence. We will say that a point $a\in \AA^1(k)$ is a \emph{critical point} of $C$ if and only if $f'(a)=0$, or $g'(b)=0$ for some $(a, b)\in C$. In other words, $a$ is a critical point if and only if there is a point of $C$ above $a$ at which one of the coordinate projections ramifies. The set of such points will be denoted by $\operatorname{Crit}(C)$.

We now recall the space of paths for the correspondence $C$. This is constructed in~\cite{corr}, but we note that the construction in the affine setting is even more direct: if $A_n$ is the coordinate ring of the affine variety parametrizing paths of length $n$, then we may simply take $\path=\operatorname{Spec}(\operatorname{colim} A_n)$, under the natural maps. In any case, for the purpose of all that follows, the reader could simply consider $\path$ to be a set of sequences as in the introduction, ignoring any geometric structure.
 
\begin{theorem}
\label{prop:paths}
Let $C$ be a correspondence on $\AA^1$. Then there exists a separated, integral, affine $k$-scheme $\path$, surjective morphisms $\pi:\path\to \AA^1$, $\epsilon:\path\to C$, and a finite surjective morphism $\sigma:\path\to \path$ making the following diagram commute.
\[\xymatrix{
\path \ar[d]_{\pi} \ar[rr]^{\sigma} \ar[dr]^{\epsilon} &  &  \path \ar[d]^\pi\\  
\AA^1 & C \ar[l]^x \ar[r]_y & \AA^1
}\]
\end{theorem}

The construction of the path space replaces our correspondence with a single-valued dynamical system $\sigma:\path\to\path$, but at some geometric cost (in the present context, $\path$ is not generally of finite type).
A path $P\in\path$ will be called \emph{preperiodic} if and only if $\sigma^n(P)=\sigma^m(P)$ for some $n\neq m$, and the point $a\in\AA^1$ will be called \emph{constrained} if there is at least one preperiodic orbit in the fibre $\path_a=\pi^{-1}(a)$ consisting of paths beginning at $a$. Note that, although the conditions of an orbit being \emph{finitely supported} or \emph{repetitive} (defined in \cite{corr}) are strictly weaker than that of being preperiodic, the existence of a repetitive orbit in $\path_a$ ensures the existence of a preperiodic one, and so the meaning of \emph{constrained} is not changed if we swap any of these terms.

The correspondence $C$ will be called \emph{post-critically constrained} (PCC) if and only if every critical point is constrained. Note that in the special case of correspondences $y=f(x)$, each fibre $\path_a$ consists of a single point, the forward orbit of $a$, and so the correspondence is PCC if and only if $f$ is post-critically finite in the usual sense. 

\begin{remark}
Counted with multiplicity, there are $de-1$ critical points of $C$, but only $d+e-2$ degrees of freedom in specifying $g$ and $f$, once the obvious changes of variables have been taken into account (see below). It might seem, then, that PCC correspodences are overconstrained when $e\geq 2$, and that there should be very few of them. However, it is easy to see that $C$ is PCC as long as $\path_a$ contains a preperiodic path whenever $f'(a)g'(a)=0$, and there are $d+e-2$ solutions to this equation. A naive dimension count, then, would suggest that there are infinitely many variable-separated PCC correspondences over any algebraically closed field.
\end{remark}

We now define the appropriate notion of equivalence of correspondences of the sort under consideration. 
Note that the curve $g(y)=f(x)$ is the same as the curve $\alpha g(y)+\beta = \alpha f(x)+\beta$, for $\alpha\neq 0$, and hence the action of the affine group on $(g, f)$ by post-composition leaves the resulting correspondence unchanged. Our notion of equivalence should reflect this. Similarly, the correspondence defined by $g(y)=f(x)$ is related by an obvious isomorphism to that defined by $g(\alpha y+\beta)=f(\alpha x+\beta)$, and so pre-composition by the affine group offers another equaivalence. 
We will say that the correspondence defined by the polynomials $(g, f)$ is \emph{equivalent} to that defined by $(G, F)$, and write $(g, f)\sim (G, F)$, if and only if there exist affine transformations $\phi, \psi$ with $\phi\circ g=G\circ\psi$ and $\phi\circ f=F\circ \psi$. Note that, since the affine linear group is not reductive, constructing a moduli space of correspondences modulo equivalence is somewhat tricky. Leaving this for future work, we content ourselves here with simply working nearly up to change of variables by choosing a suitable normal form.

\begin{lemma}
Every variable-separated polynomial correspondence with $\deg(g)<\deg(f)$ is equivalent (over an algebraically closed field) to one of the form
\begin{multline}
\label{eq:normalform}
\frac{1}{e}y^e-\frac{1}{e-1}(t_1+\cdots +t_{e-1})y^{e-1}+\cdots \pm t_1\cdots t_{e-1}y\\=\frac{1}{d}x^d-\frac{1}{d-1}(s_1+\cdots +s_{d-1})x^{d-1}+\cdots \pm s_1\cdots s_{d-1}x.
\end{multline}
\end{lemma}

\begin{proof}
Since the affine group is a group, we may construct our transformation in steps. First, note that for any point $g(a)=f(a)=A$ we may replace $(g(x), f(x))$ with $(g(x+a)-A, f(x+a)-A)$ thereby eliminating the constant terms. Now, assume that $(g, f)$ have no constant terms, and write
\[g(y)=b_ey^e+\cdots +b_1y, \qquad f(x)=a_dx^d+\cdots a_1x\]
with $b_ea_d\neq 0$. We then have
\[\alpha g(\beta y)=\alpha \beta^eb_ey^e+\cdots +\alpha \beta b_1y, \qquad \alpha f(\beta x)=\alpha\beta^da_dx^d+\cdots \alpha\beta a_1x,\]
and we choose
\[\beta^{d-e}=\frac{eb_e}{da_d}\quad\text{ and }\quad\alpha=\frac{1}{eb_e\beta^e}.\]
Replacing $(g, f)$ by $(\alpha g(\beta y), \alpha f(\beta x))$, we now have a correspondence of the form
\[\frac{1}{e}y^e+\frac{1}{e-1}b_{e-1}y^{e-1}+\cdots +b_1y = \frac{1}{d}x^d+\frac{1}{d-1}a_{d-1}x^{d-1}+\cdots +a_1x.\]
The $t_j$ and $s_i$ are obtained by factoring the derivatives of $g(y)$ and $f(x)$, respectively.
\end{proof}

The following algebraic lemma is crucial in our lower bound on the critical height.

\begin{lemma}\label{lem:nullstellensatz}
Let \[R=\ZZ\left[\frac{1}{p}:2\leq p\leq d\right],\] let $S=R[s_1, ..., s_{d-1}]$, let $f_{\mathbf{s}}(x)\in S[x]$ be defined by
\begin{equation}\label{eq:fform}
f_{\mathbf{s}}(x)=\frac{1}{d}x^d-\frac{1}{d-1}(s_1+\cdots +s_{d-1})x^{d-1}+\cdots\pm s_1\cdots s_{d-1}x,
\end{equation}
and let $F_i(s_1, ..., s_{d-1})=f_{\mathbf{s}}(s_i)\in S$. Then there exist $D\geq 1$ and $A_{ij}\in S$ such that
\begin{equation}\label{eq:homo}s_i^D=\sum_{j=1}^{d-1}A_{ij}F_j\end{equation}
for each $i$. The $A_{ij}$ may be taken to be homogeneous, of degree $D-d$.
\end{lemma}

\begin{proof}
This is essentially Lemma~7 of \cite{pcfpoly}, in slightly more algebraic terms, but we give a quick proof here for the convenience of the reader. We wish to show that the radical of the homogeneous ideal $I$ generated in $S$ by the $F_j$ is the irrelevant ideal (the ideal generated by all homogeneous elements of positive gradation), when $S$ is viewed as a graded $R$-ring in the obvious way. Suppose that this is not the case. Then the ideal generated by the $F_j$ is contained in some maximal non-irrelevant ideal $\mathfrak{m}$. Since $R$ is Jacobson, the Nullstellensatz gives us that $S/\mathfrak{m}$ is a finite extension of the field $R/(\mathfrak{m}\cap R)$, which must be a finite field $\FF_q$ of characteristic $p> d$.

In other words, there exist values $s_1, ..., s_{d-1}\in \FF_q$, not all $0$, such that $f_{\mathbf{s}}(s_i)=0$ for all $i$, where $f_{\mathbf{s}}$ is defined as in \eqref{eq:fform}. Now, if $f_{\mathbf{s}}(x)$ vanishes at $x=s_i$ then (given the characteristic of the field) it vanishes there to order
\[1+\#\{j:s_j=s_i\},\]
since this is one more than the multiplicity of $s_i$ as a root of $f_\mathbf{s}'(x)$.
Hence $f_\mathbf{s}(x)$ has at least
\[d-1+\#\{s_i:1\leq i\leq d-1\}\]
roots, counted with multiplicity, as well as a root at $x=0$ which may or may not be counted amongst the $s_i$. Since $f_\mathbf{s}$ has degree $d$, this is possible only if $s_i=0$ for all $i$, a contradiction.

We have shown that the radical of $I\subseteq S$ contains all of the $s_i$, and hence $I$ contains some power of each (without loss of generality, the same power). The fact that the $A_{ij}$ can be taken to be homogeneous of degree $D-d$ follows from noting that the homogeneous terms of degree $D$ on both sides of \eqref{eq:homo} must agree.
\end{proof}


\section{Local heights}\label{sec:greens}

In this section, we work over an algebraically closed field $k$ equipped with a valuation $v$. We will say that $v$ is \emph{$p$-adic}, for a given prime integer $p$, if and only if $0<|p|_v<1$. Note that every valuation is $p$-adic for at most one $p$, but non-archimedean valuations need not be $p$-adic for any $p$ (e.g., valuations on function fields).

Since careful attention to the error terms will be necessary in Section~\ref{sec:global}, we say a word here on terminology. As usual, if $F$ and $G$ are real-valued functions on a set $S$, we write
\begin{equation}\label{eq:leqoh}F\leq G+O(1)\end{equation} to mean that there exists a constant $c$ such that $F(x)\leq G(x)+ c$ for all $x\in S$, and 
\begin{equation}\label{eq:oh}F=G+O(1)\end{equation} to mean that $G\leq F+O(1)$ as well. In general, there may be many distinct absolute values on $k$, and so our estimates will depend on the specific nature of $v$. We will say of \eqref{eq:leqoh} that \emph{the implied constant comes from archimedean places} if and only if we may take $c=0$ when $v$ is non-archimedean, and $c$ some other absolute quantity for any archimedean absolute value (independent of the specific absolute value chosen). More generally, we will say that \emph{the implied constant comes from primes up to $d$} if and only if we may take $c$ to be absolute for archimedean and $p$-adic absolute values with $p\leq d$, and $c=0$ otherwise. When these terms are used in reference to statements of the form \eqref{eq:oh}, we simply mean the same thing for each of the implied inequalities.

In general, bounds for all of the implied constants in this section can be made explicit by an elementary (if somewhat onerous) computation, except those depending on Lemma~\ref{lem:whack}. Those constants can also be made explicit, but require some effective version of Hilbert's Nullstellensatz~\cite{mw}.

We first cite some estimates which follow from the theory of Newton polygons, when $v$ is non-archimedean, and Rouch\'e's Theorem, when $v$ is archimedean.
\begin{lemma}\label{lem:nps}
Let
\[f(x)=a_dx^d+\cdots +a_1x+a_0\in k[x],\]
with $a_d\neq 0$.
\begin{enumerate}
\item\label{it:basicnps} If $f(w)=0$, then
\[\log|w|_v\leq \log\max_{0\leq i<d}\left\{\left|\frac{a_i}{a_d}\right|_v^{1/(d-i)}\right\}+\begin{cases}\log d & \text{if } v\text{ is archimedean}\\0 & \text{otherwise.}\end{cases}\]
\item\label{it:vsmallroots} 
 For any $\kappa\geq 0$, if 
\[\log\max_{0\leq i<d}\left\{\left|\frac{a_i}{a_d}\right|_v^{1/(d-i)}\right\}\leq \log\left|\frac{a_0}{a_d}\right|_v^{1/d}+\kappa\]
then every root of $f(w)=0$ satisfies
\[\log|w|\geq \log\left|\frac{a_0}{a_d}\right|^{1/d}-(d-1)\kappa -\begin{cases} (d-1)\log d & \text{if } v\text{ is archimedean}\\0 & \text{otherwise.}\end{cases}\]
\end{enumerate}
\end{lemma}

\begin{proof}
Without loss of generality we will take $a_d=1$. Claims~\eqref{it:basicnps} and~\eqref{it:vsmallroots} both follow directly from the theory of Newton Polygons when $v$ is non-archimedean, so we will assume that $v$ is archimedean and, without loss of generality, that it is the usual absolute value on $\CC$. In this case, Claims~\eqref{it:basicnps} and~\eqref{it:vsmallroots} follow from Rouch\'{e}'s Theorem.

Specifically, let $R=\max_{0\leq i<d}|da_i|^{1/(d-i)}$, so that for $|z|=R$ we have $|a_iz^i|\leq \frac{1}{d}|z|^d$. It follows that $|f(z)-z^d|\leq |z^d|$ on the boundary of the disk of radius $R$ at the origin, and so $f(z)$ has all $d$ of its roots in this disk. That establishes that
\[\log|w|\leq\log R\leq \log\max_{0\leq i<d}\left\{\left|\frac{a_i}{a_d}\right|_v^{1/(d-i)}\right\}+\log d\]
for every root $f(w)=0$.

For \eqref{it:vsmallroots}, note that the function $F(z)=z^df(1/z)$ is holomorphic on $\CC$. Let $R=\max_{1\leq i\leq d}d^{1/i}|a_i/a_0|^{1/i}$. For $|z|=R$ we have $|a_iz^i|\leq \frac{1}{d}|a_0z^d|$ for each $1\leq i\leq d$, and so by Rouch\'{e}'s Theorem, $F(z)$ and $a_0z^d$ have the same number of roots in the disk of radius $R$ at the origin, specifically $d$. It follows that every root $f(w)=0$ satisfies $|1/w|< R$, whereupon
\[\log|w|\geq -\log R.\]
Suppose that we have 
\[\log|w|< \log|a_0|^{1/d}-\beta.\]
If the maximum defining $R$ occurs at $i=d$, then our inequalities become
\[-\frac{1}{d}\log d+\frac{1}{d}\log|a_0|=-\log R \leq \log |w|< \frac{1}{d}\log|a_0|-\beta,\]
which is plainly impossible once $\beta\geq \frac{1}{d}\log d$.

If the maximum defining $R$ is attained for some $1\leq i<d$, then we have
\[\frac{1}{i}\log\left|\frac{a_0}{a_i}\right|-\frac{1}{i}\log d< \frac{1}{d}\log|a_0|-\beta,\]
whence
\begin{eqnarray*}
\frac{1}{d}\log|a_0|+\frac{1}{d-1}\beta &\leq& \frac{1}{d}\log\left|a_0\right|+\frac{i}{d-i}\beta\\
&<& \frac{1}{d-i}\log d+\frac{1}{d-i}\log|a_i|\\
&\leq & \log d+\frac{1}{d-i}\log |a_i|\\
&\leq & \log d+\frac{1}{d}\log |a_0|+\kappa,
\end{eqnarray*}
which is a contradiction as soon as $\beta\geq (d-1)\log d+(d-1)\kappa$.
\end{proof}

For the remainder of this section, we consider a correspondence $C$ given by the equation
\begin{equation}\label{eq:generic}\sum_{j=0}^eb_jy^j=g(y)=f(x)=\sum_{i=0}^da_ix^i,\end{equation}
with $a_d\neq 0$ and $b_e\neq 0$.  Note that this form is more permissive than the form \eqref{eq:normalform} and, while we do not require the extra flexibility, it seems that the more general estimates will be useful in subsequent work. We will denote by $K$ the subfield of $k$ generated (over the prime field) by the coefficients of $C$, and by the \emph{place of $K$ under $v$} we mean the valuation of $K$ to which $v$ restricts. 

For such a correspondence, we define
\begin{multline*}
\basin(C, v)=\log^+\max\left\{\left|\frac{a_i}{a_d}\right|_v^{1/(d-i)}, \left|\frac{b_j}{b_e}\right|_v^{e/d(e-j)}, \left|\frac{b_e}{a_d}\right|_v^{e/(d-e)}:0\leq i < d, 0\leq j<e\right\}\\
+\begin{cases}\log (2d) & \text{if }v\text{ is nonarchimedean} \\ 0 & \text{otherwise}.\end{cases}
\end{multline*}

We now define an analogue of the well-known escape-rate functions in polyomial dynamics. Note that, in the present context, it does not make sense for this function to be defined on points of $\AA^1(k)$, but rather on paths in the directed graph associated to the correspondence.

\begin{definition}\label{def:greens}
For $C$ a correspondence as in \eqref{eq:generic}, and $P\in \path(k)$, let
\[G_{C}(P, v)=\lim_{n\to\infty}\left(\frac{e}{d}\right)^n\log^+|\pi\circ\sigma^n(P)|_v,\]
should this limit exists.
\end{definition}

The following lemma follows a well-known telescoping sum argument originally due to Tate, but since we are interested in fairly explicit bounds, we give an explicit proof. 

\begin{lemma}\label{lem:greensprops}
The limit defining $G_{C}(P, v)$ always exists. Furthermore,
\begin{enumerate}
\item\label{it:transform} For any path $P$,
\[G_{C}(\sigma(P), v)=\left(\frac{d}{e}\right)G_{C}(P, v).\]
\item\label{it:generalbound} For any path $P$,
\[G_{C}(P, v)=\log^+|\pi(P)|_v+O(1),\]
where the implied constant depends on $C$ and the place of $K$ above $v$.
\item\label{it:uniformpart} If $\log|\pi(P)|_v>\basin(C, v)$, then
\[G_{C}(P, v)=\log^+|\pi(P)|_v+\left(\frac{1}{d-e}\right)\log\left|\frac{a_d}{b_e}\right|_v+O(1),\]
where the implied constant comes from primes up to $d$.
\item\label{it:roughlysizebasin} For any $B\geq 1$ and $P$ with $\log|\pi(P)|\leq \basin(C, v)+B$,
we have
\[G_{C}(P, v)\leq \basin(C, v)+O(1),\]
where the implied constant depends on $e$, $d$, $B$, and the place of $K$ above $v$.
\end{enumerate}
\end{lemma}

\begin{proof}
Property \eqref{it:transform} is immediate from the definition, as long as the limit exists.

We first prove \eqref{it:uniformpart} in the case where $v$ is non-archimedean, and for convenience write $P_n=\pi\circ\sigma^n(P)$. Note that if $\log|P_0|_v>\basin(C, v)$, we must have $|f(P_0)|_v=|a_dP_0^d|_v$, since this term in the polynomial will be strictly larger than any other. On the other hand, by Lemma~\ref{lem:nps} applied to the constant term in $g(y)-f(P_0)$, we have
\[|P_1|_v=\left|\frac{a_d}{b_e}P_0^d\right|_v^{1/e}.\] Our hypotheses also ensure that this is at least as large as $|P_0|_v$, and hence we may apply the same argument to $P_n$ for all $n$, obtaining
\[|P_n|_v=|P_0|_v^{(d/e)^n}\left|\frac{a_d}{b_e}\right|_v^{1/e+d/e^2+\cdots +d^{n-1}/e^{n}}.\]
The claim now follows by taking logarithms and limits.

Now suppose that $v$ is archimedean. Then the hypothesis ensures that $|a_dP_0^d|_v>\frac{1}{2d}|a_iP_0^i|_v$, for any $i$, and hence
\[\frac{1}{2}|a_dP_0^d|_v\leq |f(P_0)|\leq \frac{3}{2}|a_dP_0^d|_v.\]
Now, applying Lemma~\ref{lem:nps} to $g(y)-f(P_0)$, we see that all roots $y$ satisfy
\[\left|\log|y|_v-\frac{d}{e}\log |P_0|_v\right|\leq c,\]
for some absolute constant $c$.
This ensures that $\log|P_1|>\basin(C, v)+c_4$, and so we may iterate this, obtaining
\[\left| \left(\frac{e}{d}\right)^{n+m}\log|P_{n+m}|-\left(\frac{e}{d}\right)^m\log|P_m|\right| \leq c\left(\frac{e}{d}\right)^m.\]
This shows that the sequence in question is Cauchy, and hence converges, and also establishes the claimed error bound.

Note that \eqref{it:uniformpart} now also implies that the limit defining $G_{C}(P, v)$ always exists, independent of the size of $\log|P_0|$. If $\log|P_n|>\basin(C, v)$ for some $n$, then
\[G_{C}(P, v)=\left(\frac{e}{d}\right)^nG_{C}(\sigma^n(P), v)=\left(\frac{e}{d}\right)^n\log^+|P_n|+O(1),\]
and if there is no such $n$ then $G_{C}(P, v)=0$.

Now, to prove \eqref{it:generalbound} and \eqref{it:roughlysizebasin}, in the case that $v$ is non-archimedean, note that for any $x$, we have by Lemma~\ref{lem:nps}
\begin{eqnarray}\label{eq:roughupper}
\log|a_dx^d+\cdots+a_0|&=&\log|a_d|+\sum_{f(w)=0}\log |x-w|\\
&\leq &\log|a_d|+\sum_{f(w)=0}\log\max\{|x|, |w|\}\\
&\leq&\log|a_d|+d\max\{\log^+|x|, \basin(C, v)\}+O(1).
\end{eqnarray}

Now, applying Lemma~\ref{lem:nps} to $g(y)-f(P_0)$ as a polynomial in $y$, we see that the largest root is of size either at most $\left|\frac{b_j}{b_e}\right|_v^{1/(e-j)}$, for some $j$, or else of size at most $|f(P_0)/b_e|^{1/e}$. In the former case we have \[\log|P_1|\leq \frac{d}{e}\basin(C, v)+O(1),\] and in the latter
\[\log|P_1|\leq \frac{1}{e}\log|f(P_0)/b_e|\leq \frac{d}{e}\max\{\log^+|P_0|, \basin(C, v)\}+\frac{1}{e}\log|a_d/b_e|+O(1).\] In either case, we have
\begin{equation}\label{eq:stepbound}\frac{e}{d}\log^+|P_1|\leq \max\{\log^+|P_0|, \basin(C, v)\}+\frac{1}{d}\log^+\left|\frac{a_d}{b_e}\right|+O(1).\end{equation}
 If $\log^+|P_0|\leq \basin(C, v)+B$, then the above becomes
\[\frac{e}{d}\log|P_1|\leq \basin(C, v)+O(1),\]
and proceeding by induction gives \eqref{it:roughlysizebasin}.

On the other hand, replacing the maximum in \eqref{eq:stepbound} by a sum gives the upper bound in \eqref{it:generalbound}. The lower bound here follows from \eqref{it:uniformpart} if $\log|P_0|$ large enough, and the non-negativity of $G_{C}$ if $\log|P_0|$ is not too large.
\end{proof}

\begin{remark}
Dinh~\cite{dinh} defines an invariant measure associated to a correspondence of the form~\eqref{eq:varsep} over $\CC$. Specifically, if $C$ is defined by $g(y)=f(x)$ and we write $C^*=x_*y^*$ (where $x$ and $y$ are the coordinate projections in \eqref{eq:varsep}), then Dinh constructs a compactly-supported probability measure $\mu$ on $\CC$ satisfying $C^*\mu=\deg(f)\mu$. Writing the potential with respect to this measure as
\[\widetilde{G}_C(z)=\int \log|z-w|d\mu(w)-I(\mu),\]
and extending this function linearly to divisors, we then have $\widetilde{G}_C(C_* z)=\deg(f)\widetilde{G}_C(z)$ and $\widetilde{G}_C(z)=\log|z|+O(1)$.

Now note, as in \cite{corr}, that the fibre $\pi^{-1}(z)\subseteq\path(\CC)$ admits a natural structure as a probability space. In particular, if we weight paths of length $n$ by their multiplicities, and take the limit of this distribution, we obtain a measure $\nu$ on $\path_z(\CC)$ with respect to which $G_C$ is measurable. If we define
\[\mathbb{E}G_C(z)=\int_{\pi(P)=z}G_C(P)d\nu(P),\]
(dropping reference to the valuation on $\CC$)
then it follows easily from Lemma~\ref{lem:greensprops} that $\mathbb{E}G_C(z)=\log|z|+O(1)$ and $\mathbb{E}G_C(C_* z)=\deg(f)\mathbb{E}G_C(z)$, where we again extend to divisors by linearity. It follows that
\[\left|\widetilde{G}_C(z)-\mathbb{E}G_C(z)\right|=\deg(f)^{-n}\left|\widetilde{G}_C(C_*^n z)-\mathbb{E}G_C(C_*^n z)\right|=O\left(\deg(f)^{-n}\right),\]
where the implied constant is independent of $n$. In particular, the potential defined by Dinh's invariant measure is precisely the average value of $G_C$ on the corresponding fibre of $\pi$.

It would be interesting to construct a measure on $\path(\CC)$ of which $G_C$ is the potential function in some reasonable sense, and show that periodic paths and preimages generally equidistribute to this measure.
\end{remark}

We now turn our attention back to correspondences specifically of the form~\eqref{eq:normalform}. Given tuples $\mathbf{s}$ and $\mathbf{t}$, we use the notation $f_{\mathbf{s}}$ and $g_{\mathbf{t}}$ as in~\eqref{eq:fform}, and write $C_{\mathbf{s}, \mathbf{t}}$ for the correspondence $g_{\mathbf{t}}(y)=f_{\mathbf{s}}(x )$. For brevity, we shall write
\[\|\mathbf{s}\|_v=\max\{|s_1|_v, ..., |s_{d-1}|_v\},\]
and similarly for $\mathbf{t}$.

\begin{lemma}\label{lem:basinboundislocalheight}
For and $\mathbf{s}, \mathbf{t}$ we have
\[\basin(C_{\mathbf{s}, \mathbf{t}}, v) = \log\max\{1, \|\mathbf{t}\|^{e/d}, \|\mathbf{s}\|\}+O(1),\]
where the implied constant comes from primes up to $d$.
\end{lemma}

\begin{proof}
Note that the coefficients of
\begin{multline*}
f_{\mathbf{s}}(x)=a_dx^d+\cdots +a_1x\\=\frac{1}{d}x^d-\frac{1}{d-1}(s_1+\cdots+s_{d-1})x^{d-1}+\cdots \pm s_1\cdots s_{d-1}x
\end{multline*}
satisfy (for $0<i<d$)
\[a_i=\left(\frac{(-1)^{d-i}}{i}\sum_{\substack{J\subseteq \{1, ..., d-1\}\\|J|=d-i}}\prod_{j\in J}s_j\right)\]
 and hence
\begin{equation}\label{eq:symbound}\log\left|\frac{a_i}{a_d}\right|_v\leq (d-i)\log\|\mathbf{s}\|_v+\log\max_{1\leq i< d}\left|\frac{d}{i}\right|_v+\begin{cases}\log\binom{d-1}{d-i} & \text{if }v\text{ is archimedean}\\ 0 & \text{otherwise.}\end{cases}\end{equation}
Similarly, we may write
\[s_i^{d-1}=\sum_{j=1}^{d-1}(-1)^{j+1}s_i^{j-1}\frac{ja_{j}}{da_d},\]
from which
\[\log\|\mathbf{s}\|_v\leq \log\max_{1\leq j<d}\left|\frac{a_j}{a_d}\right|_v^{1/(d-j)}+\log\max_{1\leq i<d}\left|\frac{j}{d}\right|_v+\begin{cases}\log d & \text{if }v\text{ is archimedean}\\ 0 & \text{otherwise.}\end{cases}\]
Applying the same reasoning to $g_{\mathbf{t}}$ completes the proof of the lemma.
\end{proof}

\begin{lemma}\label{lem:whack}
Let $C$ be a correspondence of the form \eqref{eq:normalform}. Then there exists a critical point $\sigma\in \operatorname{Crit}(C)$ such that for any $(\sigma,  \tau)\in C$ we have
\begin{equation}\label{eq:teebound}\log^+|\tau|_v=\frac{d}{e}\basin(C, v)+O(1).\end{equation}
Furthermore, the implied constant comes from primes up to $d$.
\end{lemma}

\begin{proof}
Note, first, that for \emph{any} critical point $\sigma$ and branch point $\sigma\to \tau$, we have
\[\log^+|\tau|_v\leq\frac{d}{e}\basin(C, v)+O(1).\]
Note that if $\sigma\in\operatorname{Crit}(C_{\mathbf{s}, \mathbf{t}})$, then either $\sigma=s_i$ for some $i$, or else $f_\mathbf{s}(\sigma)=g_{\mathbf{t}}(t_j)$ for some $j$.
By the triangle inequality, for any $s_i$ we have
\[\log|f_\mathbf{s}(s_i)|_v\leq d\log \|\mathbf{s}\|_v+O(1),\]
wherein the implied constant comes from primes up to $d$. Now, applying Lemma~\ref{lem:nps}  to the polynomial $g_{\mathbf{t}}(y)-f_\mathbf{s}(s_i)$ and using the analogue of \eqref{eq:symbound}, we have (for all $s_i\to\tau$)
\[\log|\tau|_v\leq \max\left\{\log\left|\frac{b_i}{b_e}\right|_v^{1/(e-i)}, \frac{d}{e}\log\|\mathbf{s}\|_v\right\}+O(1)=\frac{d}{e}\lambda(C_{\mathbf{s}, \mathbf{t}}, v)+O(1),\]
by Lemma~\ref{lem:basinboundislocalheight},
where the implied constant comes from primes up to $d$.

On the other hand, suppose that $\sigma\neq s_i$ for any $i$, so that $f_{\mathbf{s}}(\sigma)=g_{\mathbf{t}}(t_j)$ for some $j$. The branch values $\sigma\to\tau$ are the roots of $g_{\mathbf{t}}(y)-g_{\mathbf{t}}(t_j)$. Again by the triangle inequality and Lemma~\ref{lem:basinboundislocalheight} we have
\[\log|g_{\mathbf{t}}(t_j)|_v\leq e\log\|\mathbf{t}\|_v+O(1)\leq d\lambda(C_{\mathbf{s}, \mathbf{t}}, v)+O(1).\]
Applying Lemma~\ref{lem:nps}~\eqref{it:basicnps} to $g_{\mathbf{t}}(y)-g_{\mathbf{t}}(t_j)$, we have for each root $y=\tau$ that
\[\log|\tau|_v\leq \max\left\{\log\left|\frac{b_i}{b_e}\right|_v^{1/(e-i)}, \log\|\mathbf{t}\|_v\right\}+O(1)\leq \frac{d}{e}\lambda(C_{\mathbf{s}, \mathbf{t}}, v)+O(1)\]
as above.

 The challenge, of course, it to show that there exists a $\sigma\in \operatorname{Crit}(C_{\mathbf{s}, \mathbf{t}})$ such that this upper bound is more-or-less realized for all $\sigma\to\tau$.
By Lemma~\ref{lem:nullstellensatz} and the triangle inequality, we have some $t_j$ with 
\[\log|g_{\mathbf{t}}(t_j)|_v\geq e\log\|\mathbf{t}\|_v-C_{1, v},\] where we may take $C_{1, v}=0$ except possibly for primes $v$ up to $d$. Note that we have, for $g_{\mathbf{t}}(z)=b^ez^e+\cdots +b_1z$, that 
\begin{equation}\label{eq:beejays}\log|b_i/b_e|_v^{1/(e-i)}\leq \log\|\mathbf{t}\|_v+C_{2, v},\end{equation}
where again we may take $C_{2, v}=0$ except possibly for primes $v$ up to $d$. If we set $b_0=-g_{\mathbf{t}}(t_j)$, the polynomial \[g_{\mathbf{t}}(z)-g_{\mathbf{t}}(t_j)=b_ez^e+\cdots+b_0\] then satisfies the condition of Lemma~\ref{lem:nps}~\eqref{it:vsmallroots} with \[\kappa=\max\{0, e^{-1}C_{1, v}+e^{-1}\log|e|_v+C_{2, v}\}.\] We see that every root $y=\tau$ of $g_{\mathbf{t}}(y)-g_{\mathbf{t}}(t_j)$ then satisfies \[\log|\tau|_v\geq\log\|\mathbf{t}\|_v-C_{3, v},\]
with $C_{3, v}=0$ except possibly when $v$ is a prime up to $d$. Let $C_{4, v}$ be a constant to be chosen later, with the stipulation that $C_{4, v}=0$ except for primes $v$ up to $d$.
If $\log\|\mathbf{t}\|_v\geq \frac{d}{e}\log^+\|\mathbf{s}\|_v-C_{4, v}$, then we are done, by Lemma~\ref{lem:basinboundislocalheight}, selecting $\sigma$ to be any root of $f_{\mathbf{s}}(\sigma)=g_{\mathbf{t}}(t_j)$. 

We are left with the case that $\log\|\mathbf{t}\|_v+C_{4, v}<\frac{d}{e}\log^+\|\mathbf{s}\|_v$, and we shall for now assume that $\|\mathbf{s}\|_v\geq 1$.
Again by Lemma~\ref{lem:nullstellensatz} and the triangle inequality, there exists an $s_i$ with
\[\log|f_\mathbf{s}(s_i)|_v\geq d\log\|\mathbf{s}\|_v-C_{5, v},\]
for $C_{5, v}$ a constant with the usual confinements. Consider the polynomial
$g_{\mathbf{t}}(y)-f_{\mathbf{s}}(s_i)$. We note that the estimate~\eqref{eq:beejays} still holds for $1\leq i<e$, and now we also have
\[\log\left|\frac{-f_{\mathbf{s}}(s_i)}{b_e}\right|_v^{1/e}\geq\frac{d}{e}\log\|\mathbf{s}\|_v-\frac{C_{5, v}}{e}+\frac{\log|e|_v}{e}\]
yielding
\begin{equation}\label{eq:this}
\log\left|\frac{-f_{\mathbf{s}}(s_i)}{b_e}\right|_v^{1/e}\geq \log\left|\frac{b_i}{b_e}\right|_v^{1/(e-i)}-C_{2, v}+C_{4, v}-\frac{C_{5, v}}{e}+\frac{\log|e|_v}{e}.
\end{equation}
In other words, now choosing
\[C_{4, v}=C_{2, v}+\frac{C_{5, v}}{e}-\frac{\log|e|_v}{e},\]
we may apply Lemma~\ref{lem:nps}~\eqref{it:vsmallroots} to show that every root $y=\tau$ of $g_{\mathbf{t}}(y)-f_{\mathbf{s}}(s_i)$ satisfies
\[\log|\tau|_v\geq  \frac{d}{e}\log\|\mathbf{s}\|_v-\frac{C_{5, v}}{e}+\frac{\log|e|_v}{e}+\begin{cases} (d-1)\log d & v\text{ archimedean}\\0 & \text{otherwise.}\end{cases}\]
We are again done, by Lemma~\ref{lem:basinboundislocalheight}, unless we have $\|\mathbf{s}\|_v<1$.

In this final case, though, Lemma~\ref{lem:basinboundislocalheight} gives $\lambda(C_{\mathbf{s}, \mathbf{t}}, v)=O(1)$, for an appropriate constant, and from this we deduce the lower bound implied in \eqref{eq:teebound} from the non-negativity of $\log^+|\tau|_v$.
\end{proof}

We now define the local contribution to what will be the critical height. First, let
\[\mathscr{C}=\prod_{\sigma\in \operatorname{Crit}(C)}\path_\sigma,\]
so that an element of $\mathscr{C}$ corresponds to a particular choice of path for each critical point of $C$. Set
\[\Lambda(C, v)=\inf_{(P_1, ..., P_{de-1})\in\mathscr{C}}\max\left\{G_{C}(P_1, v), ..., G_{C}(P_{de-1}, v)\right\},\]
which is slightly different from the function $\Lambda$ mentioned in the introduction. 
Most of the content of this paper follows from the estimate in the next lemma.

\begin{lemma}\label{lem:mainlocal}
For $C$ in the form \eqref{eq:normalform} we have
\begin{equation}\label{eq:themainbound}\Lambda(C, v)=\basin(C, v)+O(1),\end{equation}
where the implied constant comes from primes up to $d$.
\end{lemma}

\begin{proof}
By Lemma~\ref{lem:basinboundislocalheight},  it suffices to relate $\Lambda(C_{\mathbf{s}, \mathbf{t}}, v)$ to $\log\max\{1, \|\mathbf{s}\|, \|\mathbf{t}\|^{e/d}\}$.
Note that the implied upper bound on $\Lambda(C_{\mathbf{s}, \mathbf{t}}, v)$ is quite easy, given what we have shown so far. In particular, we have
\[\log|s_i|_v\leq \basin(C_{\mathbf{s}, \mathbf{t}}, v)+O(1)\]
for any $i$, and hence Lemma~\ref{lem:greensprops}~\eqref{it:roughlysizebasin} gives that
\[G_{C_{\mathbf{s}, \mathbf{t}}}(P, v)\leq \basin(C_{\mathbf{s}, \mathbf{t}}, v)+O(1)\]
for any $P\in \path_{s_i}$. On the other hand, if $\sigma\to t_j$, then Lemma~\ref{lem:nps} and the triangle inequality give $\log|\sigma|_v\leq \basin(C_{\mathbf{s}, \mathbf{t}}, v)+O(1)$, and again we have the same bound on $G_{C_{\mathbf{s}, \mathbf{t}}}(P, v)$ for any $P\in \path_\sigma$. This establishes that $G_{C_{\mathbf{s}, \mathbf{t}}}(P, v)\leq \basin(C_{\mathbf{s}, \mathbf{t}}, v)+O(1)$ for any $P\in\path_\sigma$, for any $\sigma\in\operatorname{Crit}(C_{\mathbf{s}, \mathbf{t}})$, which is the required upper bound. It remains to show that this bound is attained.

By Lemma~\ref{lem:whack}, there exists a $\sigma\in \operatorname{Crit}(C_{\mathbf{s}, \mathbf{t}})$ such that for every $\sigma\to\tau$ we have
\[\log^+|\tau|_v\geq \frac{d}{e}\basin(C_{\mathbf{s}, \mathbf{t}}, v)-C_{6, v}.\]
Suppose at first that $C_{6, v}<\frac{d-e}{e}\basin(C_{\mathbf{s}, \mathbf{t}}, v)$.
For such a $\sigma$, and any path $\pi(P)=\sigma$, Lemma~\ref{lem:greensprops}~\eqref{it:uniformpart} now gives
\begin{eqnarray*}
G_{C_{\mathbf{s}, \mathbf{t}}}(P, v)&=&\frac{e}{d}G_{C_{\mathbf{s}, \mathbf{t}}}(\sigma(P))\\
&=&\frac{e}{d}\log^+|\tau|_v+\left(\frac{e}{d(d-e)}\right)\log\left|\frac{e}{d}\right|_v+O(1)\\
&=&\basin(C_{\mathbf{s}, \mathbf{t}}, v)+O(1),
\end{eqnarray*}
where the implied constant comes from primes up to $d$. Since \[\inf_{\pi(P)=\sigma}G_{C_{\mathbf{s}, \mathbf{t}}}(P, v)\leq \Lambda(C_{\mathbf{s}, \mathbf{t}}, v),\] we have our lower bound.

Of course, if $C_{6, v}\geq \frac{d-e}{e}\basin(C_{\mathbf{s}, \mathbf{t}}, v)$, then the lower bound implicit in~\eqref{eq:themainbound} follows from the non-negativity of $\Lambda(C_{\mathbf{s}, \mathbf{t}}, v)$.
\end{proof}

We now have the necessary material to prove the geometric result.

\begin{proof}[Proof of Theorem~\ref{th:thurston}]
Let $C/X$ be a family as in the statement of the result. After a change of variables, and possibly an extension of the base, we may assume that the correspondences are given in the form $g_{\mathbf{t}}(y)=f_{\mathbf{s}}(x)$, and that $X$ is a curve. Now, if there is some point on the projective closure of $X$ at which some $s_i$ or some $t_j$ is not regular, consider the family as a correspondence over the local ring at this point, with valuation $v$. By Lemma~\ref{lem:mainlocal}, noting that all error terms vanish for $v$, we have $\Lambda(C_{\mathbf{s}, \mathbf{t}}, v)>0$. This is impossible if the correspondence is PCC. So it follows that the $s_i$ and the $t_j$ are regular at every point of the normalization of the projective closure of $X$. In particular, they are all constant, whereupon $C$ is isotrivial.
\end{proof}

\begin{remark}
In the next section we prove various height inequalities over number fields. It is not hard to modify these arguments to work in the case of function fields, in which case we obtain a lower bound for $h_{\mathrm{Crit}}(C)$ in terms of the degrees of the coefficients. 
\end{remark}

\begin{lemma}\label{lem:cts}
The function $k^{d+e-2}\to \RR$ defined by $(\mathbf{s}, \mathbf{t})\mapsto \Lambda(C_{\mathbf{s}, \mathbf{t}}, v)$ is $v$-adically continuous.
\end{lemma}

\begin{proof}
For each $n\geq 1$, consider the function $\ell_n:k^{d+e-2}\to \RR$ defined by
\[\ell_n(\mathbf{s}, \mathbf{t})=\left(\frac{e}{d}\right)^n\inf_{\mathbf{P}\in\mathscr{C}(C_{\mathbf{s}, \mathbf{t}})}\max\left\{\log^+|\pi\circ\sigma^n(P_1)|_v, ..., \log^+|\pi\circ\sigma^n(P_{de-1})|_v\right\}.\]
Note that, by the continuity of roots of a polynomial in the coefficients, each $\ell_n$ is a continuous function.

On the other hand, we claim that $\Lambda(C_\cdot, v)$ is the uniform limit of the $\ell_n$ on any bounded subset of $k^{d+e-2}$. 
Note that for any $\mathbf{P}\in\mathscr{C}(C_{\mathbf{s}, \mathbf{t}})$ we have (by Lemma~\ref{lem:nps}~\eqref{it:uniformpart} and~\eqref{it:roughlysizebasin})
\begin{multline*}
\left(\frac{e}{d}\right)^n\max\left\{\log^+|\pi\circ\sigma^n(P_1)|_v, ..., \log^+|\pi\circ\sigma^n(P_{de-1})|_v\right\}\\ = \max\left\{G_{C}(P_1, v), ..., G_{C}(P_{de-1}, v)\right\}+O\left(\left(\frac{e}{d}\right)^n\basin(C_{\mathbf{s}, \mathbf{t}}, v)\right),
\end{multline*}
where the implied constant is independent of the choice of paths, and of $(\mathbf{s}, \mathbf{t})$. Taking the infimum over all $\mathbf{P}\in\mathscr{C}$ thus preserves the inequalities, and so we have
\[\ell_n(\mathbf{s}, \mathbf{t})=\Lambda(C_{\mathbf{s}, \mathbf{t}}, v)+O\left(\left(\frac{e}{d}\right)^n\basin(C_{\mathbf{s}, \mathbf{t}}, v)\right).\]
This shows that the $\ell_n$ converge uniformly on bounded subsets to $\Lambda(C_\cdot, v)$, and so the latter is continuous.
\end{proof}

This, in particular, allows us to prove Theorem~\ref{th:mandelbrot}.
\begin{proof}[Proof of Theorem~\ref{th:mandelbrot}]
Working over $k=\CC$ with $v$ the usual absolute value, we have $S_{d, e}$ defined by $\Lambda(C_{\mathbf{s}, \mathbf{t}}, v)=0$, and hence $S_{d, e}$ is closed. But by Lemmas~\ref{lem:basinboundislocalheight} and~\ref{lem:mainlocal}, $S_{d, e}$ is also contained in a compact subset of $\CC^{d+e-2}$.
\end{proof}


\section{Global heights}\label{sec:global}

As discussed in \cite{corr}, the global theory of correspondences is in one fundamental way different from the global theory of (deterministic) dynamical systems. In particular, and variety $X$ over a number field $K$ has the property that every $\overline{K}$-point on $X$ is an $L$-point for some number field $L$. It thus generally suffices to sort out the global theory over a number field.

In the present context, this is not the case. Since $\path$ is typically not a scheme of finite type over the ground field, there will be many $\overline{K}$-points which are not rational over any finite extension of $K$. For this reason, we must fundamentally work over $\overline{K}$. As noted in \cite{corr}, this turns the familiar weighted sums defining global heights into integrals over the space of absolute values on $\overline{K}$, as per Gubler's theory of $M$-fields. We recall the following.
\begin{theorem}
Let $K$ be a number field, and let $M_{\overline{K}}$ be the set of absolute values on $\overline{K}$ extending those on $K$. Then there is a $\sigma$-algebra $\mathcal{B}$ of subsets of $M_{\overline{K}}$, and a measure $\mu$ such that the measure space $(M_{\overline{K}}, \mathcal{B}, \mu)$ is $\sigma$-finite, and for any finite extension $L/K$, and any absolute value $v$ on $L$, we have \[A_{v, L}=\{w\in M_{\overline{K}}: w\mid v\}\in \mathcal{B}\]
and
\begin{equation}\label{eq:localdegree}\mu(A_{v, L})=\frac{[L_v:K_v]}{[L:K]}.\end{equation}
Moreover, for any $\alpha\in \overline{K}$, the function $v\mapsto \log^+|\alpha|_v$ is measurable, and
\[h(\alpha)=\int_{M_{\overline{K}}} \log^+|\alpha|_vd\mu(v).\]
\end{theorem}

Before proceeding, we must say something about the integration of error terms. Supose that $F$ and $G$ are both measurable functions on $M_{\overline{K}}$, and that $F=G+O(1)$, where the implied constant depends on $v$. Then the difference $|F-G|$ is a bounded, measurable function, but it is not clear that its integral is finite. For this, we need the difference to be bounded by an element of $L^1(M_{\overline{k}})$, and since our space has infinite measure there are certainly bounded measurable functions which are not $L^1$ (integrable). Fortunately, the constants in Section~\ref{sec:greens} can all be represented by $L^1$ functions on $M_{\overline{k}}$. Specifically, let $\epsilon:M_{\overline{K}}\to \RR$ be a function with the property that for some finite $L/K$, $\epsilon$ is constant on sets of the form $A_{w, L}$, and moreover vanishes on all but finitely many sets of this form. Then $\epsilon$ is $L^1$, and in fact
\[\int_{M_{\overline{K}}}\epsilon(v)d\mu(v)=\sum_{w\in M_L}\frac{[L_w:K_w]}{[L:K]}\epsilon(w).\]
All of the error terms in Section~\ref{sec:greens} admit upper bounds of this form.

We now define, for a path $P\in\path$,
\[\hat{h}_C(P)=\lim_{n\to\infty}\left(\frac{e}{d}\right)^{-1}h(\pi\circ\sigma^n(P)).\]
It is a consequence of Theorem~1.1 of~\cite{corr} that this limit always exists, although we give an alternate proof in this special case.
\begin{proposition}
For fixed $C$ and $P\in\path$, the function $G_C(P, \cdot)$ is measurable, and we have
\[\hat{h}_C(P)=\int_{M_{\overline{K}}}G_{C}(P, v)d\mu(v).\]
\end{proposition}

\begin{proof}
First we show that $v\mapsto G_{C}(P, v)$ is measurable. To see this, note that the function $v\mapsto \log^+|\pi\circ\sigma^n(P)|_v$ is constant on the set of valuations restricting to a given valuation on the number field $K(\pi\circ\sigma^n(P))$. Since these sets are measurable, and cover $M_{\overline{K}}$, the function
\[v\mapsto \left(\frac{e}{d}\right)^n\log^+|\pi\circ\sigma^n(P)|_v\]
is measurable for each $n$. The function $G_{C}(P, \cdot)$ is thus exhibited as a pointwise limit of measurable functions, and hence is itself measurable.

Now let
\[\tilde{h}(P)=\int_{M_{\overline{K}}}G_{C}(P, v)d\mu(v),\]
which \emph{a priori} might be infinite.
From Lemma~\ref{lem:greensprops}, we have
\[G_{C}(P, v)=\log^+|\pi(P)|_v+O(1),\]
where the implied constant depends only on the restriction of $v$ to $K$, and vanishes outside of a finite set of places. In particular, the error term may be taken to be an $L^1$ function independent of $P$.
Intergrating both sides then shows that the integral defining $\tilde{h}(P)$ is finite, and gives
\[\tilde{h}(P)=h(\pi(P))+O(1),\]
where the constant does not depend on $P$.
On the other hand, it also follows from Lemma~\ref{lem:greensprops} that $\tilde{h}(\sigma(P))=\left(\frac{e}{d}\right)\tilde{h}(P)$, and so
\[\tilde{h}(P)=\left(\frac{e}{d}\right)^{n}h(\pi\circ\sigma^n(P))+o(1),\]
where $o(1)\to 0$ as $n\to\infty$. This proves that
\[\int_{M_{\overline{K}}}G_{C}(P, v)d\mu(v)=\tilde{h}(P)=\lim_{n\to\infty}\left(\frac{e}{d}\right)^{n}h(\pi\circ\sigma^n(P)),\]
and so in particular that the limit on the right exists.
\end{proof}

We will construct our critical height in a similar fashion, from local contributions, but first we must check that it makes sense to integrate these contributions.

\begin{lemma}
The function $\Lambda(C, \cdot)$ is $M_{\overline{K}}$-integrable.
\end{lemma}

\begin{proof}
We exhibit the function in question as a pointwise limit of measurable functions, along the lines of the proof of Lemma~\ref{lem:cts}.
For each $n\geq 1$, consider the function
\[\ell_n(v)=\left(\frac{e}{d}\right)^n\inf_{\mathbf{P}\in\mathscr{C}}\max\left\{\log^+|\pi\circ\sigma^n(P_1)|_v, ..., \log^+|\pi\circ\sigma^n(P_{de-1})|_v\right\}.\]
Note that there are only finitely many possibilities for $\pi\circ\sigma^n(P_i)$, for each $i$, and hence all are contained in some number field $L_n/K$. The function $\ell_n$ is now constant on the set of absolute values with a given restriction to $L_n$, and hence is measurable. 

On the other hand, we claim that $\Lambda(C, \cdot)$ is the pointwise limit of the $\ell_n$. 
Note that for any tuple of paths $\mathbf{P}=(P_1, ..., P_{de-1})\in\mathscr{C}$ we have
\begin{multline*}
\left(\frac{e}{d}\right)^n\max\left\{\log^+|\pi\circ\sigma^n(P_1)|_v, ..., \log^+|\pi\circ\sigma^n(P_{de-1})|_v\right\}\\ = \max\left\{G_{C}(P_1, v), ..., G_{C}(P_{de-1}, v)\right\}+O\left(\left(\frac{e}{d}\right)^n\right),
\end{multline*}
where the implied constant is independent of the choice of paths. Taking the infimum over all $\mathbf{P}\in\mathscr{C}$ thus preserves the inequalities, and so we have
\[\ell_n(v)=\Lambda(C, v)+O\left(\left(\frac{e}{d}\right)^n\right).\]
This shows both that $\lim_{n\to\infty} \ell_n(v)$ exists, for each $v\in M_{\overline{K}}$, and that it is equal to $\Lambda(C, v)$.

To confirm that $\Lambda(C, v)$ is in fact integrable, it suffices to note that there is a finite set of places of $K$ outside of which $G_C(P_i, v)=0$ for every $i$.
\end{proof}

We now define
\[h_{\mathrm{Crit}}(C)=\int_{v\in M_{\overline{K}}}\Lambda(C, v)d\mu(v).\]
Note that for any choices of paths $P_c\in\path_c$ for $c\in\operatorname{Crit}(C)$, we have
\begin{eqnarray*}
h_{\mathrm{Crit}}(C)&=&\int_{M_{\overline{K}}} \Lambda(C, v)d\mu(v)\\
&\leq &\int_{M_{\overline{K}}} \max_{c\in \operatorname{Crit}(C)}\left\{G_{C}(P_c, v)\right\} d\mu(v)\\
&\leq &\int_{M_{\overline{K}}} \sum_{c\in \operatorname{Crit}(C)}G_{C}(P_c, v) d\mu(v)\\
&=& \sum_{c\in\operatorname{Crit}(C)}\hat{h}_C(P_c),
\end{eqnarray*}
and hence PCC correspondences must satisfy $h_{\mathrm{Crit}}(C)=0$.

\begin{proof}[Proof of Theorem~\ref{th:main}]
Recall that by Lemma~\ref{lem:mainlocal}, we have (fixing the bidegree of $C$)
\[\Lambda(C, v)=\basin(C, v)+O(1),\]
where the implied constant is independent of $C$, and comes from primes up to $d$. In particular, the error can be bounded by a function of $v$ that is $M_{\overline{K}}$-integrable.

Let $L=K(a_i, b_j)$, where $C$ is the correspondence defined by polynomials with these coefficients.
By definition, we have
\begin{eqnarray*}
h_{\mathrm{Crit}}(C)&=&\int_{M_{\overline{K}}} \Lambda(C, v)d\mu(v)\\
&=&\int_{M_{\overline{K}}} \left(\basin(C, v)+O(1)\right)d\mu(v)\\
&=&\sum_{v\in M_L} \frac{[L_v:K_v]}{[L:K]}\basin(C, v) + O(1)\\
&=&h_{\mathrm{Weil}}(C)+O(1),
\end{eqnarray*}
where the implied constant depends only on the bidegree. Note the the height $h_{Weil}$ here is the Weil height on the weighted projective completion of \[\operatorname{Spec}(\QQ[a_{d-1}, ..., a_1, b_{e-1}, ..., b_1])\] corresponding to the assignment of weight $d-i$ to $a_i$, and $d(e-j)/e$ to $b_j$.
\end{proof}

\begin{proof}[Proof of Theorem~\ref{th:heightbound}]
Theorem~\ref{th:heightbound} follows easily from Theorem~\ref{th:main} and the proof of Theorem~\ref{th:thurston}, although we must note the slightly different normal forms.

First, suppose that $C$ is of the form \eqref{eq:normal}, with $a_i, b_j\in \CC$. We may then change coordinates so that $C$ takes the form $g_{\mathbf{t}}(y)=f_{\mathbf{s}}(x)$ with $(\mathbf{s}, \mathbf{t})\in \CC^{d+e-2}$. Note that the entries of the tuples $\mathbf{s}$ and $\mathbf{t}$ are contained in some finitely generated extension $\overline{\QQ}$, which is isomorphic to the function field of some variety $X/\overline{\QQ}$. By the proof of Theorem~\ref{th:thurston}, we have a contradiction unless the $s_i$ and $t_j$ all represent constant functions on $X$. In other words, while \emph{a priori} $s_i, t_j\in \CC$, we obtain PCC correspondences only when $s_i, t_j\in\overline{\QQ}$. We may now restriction attention to correspondences defined over $\overline{\QQ}$.

As noted above, if $C$ is PCC then $h_{\mathrm{Crit}}(C)=0$, and hence Theorem~\ref{th:main} gives a bound on $h_{\mathrm{Weil}}(C)$ in terms of the bidegree. For any coefficients $a_i$, $b_j$, we then have
\[h(a_i), h(b_j)\leq dh_{\mathrm{Weil}}(C)+O(1),\]
which can be established by looking at the local contributions.
\end{proof}


\section{Unicritical correspondences}

As noted in the introduction, it is not immediately clear that there are infinitely many PCC correspondences of given bidegree over $\CC$ (or, indeed, over any particular algebraically closed field). Given the normal form \eqref{eq:normalform}, which marks certain critical and branch points, and integers $m_i<n_i$, it is easy to write down a system of $e+d-2$ equations any solution of which gives a correspondence admitting paths $P_1, ..., P_{e+d-2}$ satisfying
\[\sigma^{m_i}(P_i)=\sigma^{n_i}(P_i),\]
$\pi(P_i)=s_i$ for $1\leq i\leq d-1$, and $\pi(P_{d-1+j})=t_j$ for $1\leq j\leq e-1$. It is not obvious, though, that these equations always have solutions and, more subtly, it is not obvious that there are infinitely many distinct solutions as one varies the $m_i$ and $n_i$. \emph{A priori}, it is possible that for any choice of integers, the only solution to the above system is the trivial one (or more generally, that there are only finitely many solutions, occuring with very high multiplicity as one increases the parameters).

It would be of interest to establish that this pathology does not arise, but the situation is genuinely more intricate than in the case of single-valued dynamics. It is a consequence of Thurston's rigidity theorems~\cite{thurston} that the varieties cut out by the equations discussed above are non-singular for correspondences of the form $y=f(x)$, but this turns out not be the case in the present setting (a claim evident from the calculations below).
 
In this section we prove that there are infinitely many distinct $PCC$ correspondences of bidegree $(p, e)$, with $p>e$ a prime. In particular, we consider correspondences of the form
\begin{equation}\label{eq:mordell}C:y^e=x^p+c,\end{equation}
a natural generalization of unicritical polynomial dynamical systems. Note that it is not the case that the correspondence $C$ defined on $\AA^1$ has only one critical point, given the definition in Section~\ref{sec:geom}, but it \emph{is} true that every $a\in\operatorname{Crit}(C)$ satisfies either $a=0$ or $a^p+c= 0$, and hence $C$ is PCC if and only if there is a preperiodic path begining at $0$. For this reason, we abuse terminology somewhat and refer to these correspondences as \emph{unicritical}.

As remarked in the introduction, we will prove something stronger than Theorem~\ref{th:notallobvious}.
\begin{theorem}\label{th:fp}
For all but finitely many $n\geq 1$, there exists a $c\in\overline{\FF_p}$ such that the  correspondence defined by $y^e=x^p+c$ has a critical path of length $n$, and no critical path of length $m$ for any $m<n$.
\end{theorem}

Note that the claim in the theorem is stronger than the claim that the correspondence has a critical path of length exactly $n$. Since the dynamics is multi-valued, it is imaginable that this correspondence has two critical paths of different (finite) lengths.

\begin{proof}[Proof of Theorem~\ref{th:notallobvious} given Theorem~\ref{th:fp}]
We define a sequence of polynomials $R_n(w, c)\in \ZZ[w, c]$ by $R_0(w, c)=w$ and
\[R_{n+1}(w, c)=\operatorname{Res}_z(R_n(z, c), z^p+c-w^e).\]
Note that $R_n(w, c)=0$ if and only if there is a path of length $n$ in the correspondence defined by \eqref{eq:mordell} originating at $0$ and terminating at $w$. In particular, if we set
\[f_n(c)=R_n(0, c),\]
then $f_n(c)=0$ if and only if $0$ is contained in a cycle of length dividing $n$ for the correspondence defined by \eqref{eq:mordell}. Note that the $f_n(c)\in\ZZ[c]$ are monic.

Now, if Theorem~\ref{th:notallobvious} fails, then there are only finitely many distinct solutions to $f_n(c)=0$ as $n$ varies. These solutions must all be algebraic, contained say in the number field $K/\QQ$. If $\pf$ is a prime of $K$ dividing $p$, then reduction-modulo-$\pf$ maps the solutions of $f_n(c)=0$ in $K$ surjectively to the solutions in $\Ocal_K/\pf\Ocal_K$ (where $\Ocal_K$ is the ring of integers of $K$), and hence there are only finitely many distinct solutions to $f_n(c)=0$ in $\overline{\FF_p}$ as $n$ varies. But then there exists a finite bound $B$ such that if $c\in\overline{\FF_p}$ and $y^e=x^p+c$ has a periodic critical path, it has one of length at most $B$. This contradicts Theorem~\ref{th:fp}.
\end{proof}

For the remainder of this section, we work over $\FF_p$, and we will consider the image of the polynomials $f_n$ defined in the proof of Theorem~\ref{th:notallobvious} in $\FF_p[c]$. Our proof comes down to showing that, for $n$ large enough, the polynomial $f_n$ has an irreducible factor which is not a factor of $f_m$ for any $m<n$, that is, a \emph{primitive prime factor}. If $c\in\overline{\FF_p}$ is a root of this primitive prime factor, then there will be cycle in the correspondence \eqref{eq:mordell} containing 0 which is of length dividing $n$, but no cycle of length dividing $m$ for any $m<n$. In other words, the cycle will have length precisely $n$, and there will be no smaller cycle containing $0$.

\begin{lemma}
For all $n\geq 1$,
\[f_{n+1}(c)=c^{p^{n}}- f_n(c)^e.\]
\end{lemma}

\begin{proof}
We will prove that
\[R_{n+1}(w, c)=c^{p^{n}}- R_n(w, c)^e,\]
from which the lemma will follow by specializing at $w=0$.

It is easy to check from the definition that
\[R_0(w, c)=w\] and so
\[R_{0+1}(w, c)=\operatorname{Res}_z(z, z^p+c-w^e)=c-w^e=c^{p^0}-R_0(w, c)^e.\]
We proceed from here by induction. By definition,
\[R_{n+1}(w, c)=\prod_{\alpha^p=w^e-c}R(\alpha, c),\]
where we take $\alpha\in\overline{\FF_p(w, c)}$. But there is a unique such $\alpha$, which occurs with multiplicity $p$, and hence we have
\begin{eqnarray*}
R_{n+1}(w, c)&=&R_n(\alpha, c)^p\\
&=&\left(c^{p^{n-1}}-R_{n-1}(\alpha, c)^e\right)^p\\
&=&c^{p^{n}}-\left(R_{n-1}(\alpha, c)^p\right)^e\\
&=&c^{p^{n}}-\operatorname{Res}_z(R_{n-1}(z, c), z^p+c-w^e)^e\\
&=&c^{p^{n}}-R_n(w, c)^e.
\end{eqnarray*}
\end{proof}

\begin{lemma}
Let $\pf\subseteq \FF_p[c]$ be a prime ideal, and let $r$ be the least integer such that $f_r\in \pf$, if one exists. Then for all $n$,
\[v_\pf(f_n)=\begin{cases}
 e^{k-1}v_\pf(f_r) & n=rk\\
 0 & r\nmid n
\end{cases}
\]
\end{lemma}

\begin{proof}
By hypothesis, we have $f_r\equiv 0\MOD{\pf}$, but $f_j\not\equiv 0\MOD{\pf}$ for all $1\leq j< r$. Now, by the above lemma, we have
\[f_{r+1}=c^{p^{r}}-f_r^e\equiv c^{p^{r}}\equiv f_1^{p^{r}}\MOD{\pf},\]
and similarly
\[f_{r+2}=c^{p^{r+1}}-f_{r+1}^e\equiv c^{p^{r+1}}- f_1^{p^{r}}\equiv\left(c^p-f_1\right)^{p^{r}}\equiv f_2^{p^{r}}\MOD{\pf}.\]
Indeed, proceeding by induction, it is easy to see that $f_{r+j}\equiv f_j^{p^{r}}\MOD{\pf}$ for all $j\geq 0$. It follows that $v(f_n)=0$ unless $r\mid n$.

Now let $\pi\in\FF_p[c]$ be an irreducible polynomial generating $\pf$, and write
\[f_j=b_j+O(\pi)\text{ for }1\leq j<r\]
and
\[f_r=A\pi^{m}+O\left(\pi^{m+1}\right),\]
where $A, b_j\not\in\pf$. We will prove by induction on $k$ that
\[v(f_{kr})=e^{k-1}m,\]
which clearly holds in the case $k=1$.

Now suppose that the claim holds for $k=n$, and write
\[f_{nr}=a_1\pi^{me^{n-1}}+O\left(\pi^{me^{n-1}+1}\right),\]
with $a_1\not\in \pf$.  As above, we have
\[f_{nr+1}=c^{p^{nr}}-f_{nr}^e=f_1^{p^{nr}}-a^e\pi^{me^n}+O\left(\pi^{me^n+1}\right).\]
Similarly,
\begin{eqnarray*}
f_{nr+2}&=&c^{p^{nr+1}}-f_{nr+1}^e\\
&=&c^{p^{nr+1}}-\left(f_1^{p^{nr}}-a_1^e\pi^{me^n}+O\left(\pi^{me^n+1}\right)\right)^e\\
&=&c^{p^{nr+1}}-f_1^{ep^{nr}}+eb_1^{e-1}a_1^e\pi^{me^n}+O\left(\pi^{me^n+1}\right)\\
&=&f_2^{p^{nr}}-a_2\pi^{me^n}+O\left(\pi^{me^n+1}\right)\\
\end{eqnarray*}
with $a_2\not\in\pf$. Proceeding in this fashion, we have
\[f_{nr+r}=f_r^{p^{nr}}-a_r\pi^{me^n}+O\left(\pi^{me^n+1}\right),\]
with $a_r\not\in\pf$. Since $f_r^{p^{nr}}=O\left(\pi^{mp^{nr}}\right)$, and since $mp^{nr}\geq me^n+1$, we have
\[f_{nr+r}=-a_r\pi^{me^n}+O\left(\pi^{me^n+1}\right).\]
By induction, the result follows.
\end{proof}

\begin{lemma}
For all $n$ sufficiently large, there is a prime ideal $\pf\subseteq\FF_p[c]$ with $f_n\in \pf$ but $f_j\not\in \pf$ for all $1\leq j<n$.
\end{lemma}

\begin{proof}
Suppose that $n$ is a value for which the claim fails, that is, for which we have $f_j\in \pf$ for some $1\leq j<n$ whenever $\pf\subseteq\FF_p[c]$ is a prime ideal with $f_n\in\pf$. Our aim is to bound $n$.

If $r_\pf$ is the index of the first polynomial divisible by $\pf$, our hypothesis gives
\[v_{\pf}(f_n)=e^{\frac{n}{r_\pf}-1}v_{\pf}(f_{r_\pf})\]
for each $\pf\mid f_n$. Observe that
\begin{eqnarray}
p^{n-1}&=&\deg(f_n)\nonumber \\
&=&\sum_{\pf\mid f_n}v_{\pf}(f_n)\deg(\pf)\nonumber \\
&=&\sum_{\pf\mid f_n}e^{\frac{n}{r_\pf}-1}v_{\pf}(f_{r_\pf})\deg(\pf)\nonumber \\
&=&\sum_{\substack{s\mid n\\ s\neq 1}}e^{s-1}\sum_{r_\pf=n/s}v_{\pf}(f_{r_\pf})\deg(\pf)\nonumber \\
&\leq &\sum_{\substack{s\mid n\\ s\neq 1}}e^{s-1}\deg(f_{n/s})\nonumber \\
&=&\sum_{\substack{s\mid n\\ s\neq 1}}e^{s-1}p^{\frac{n}{s}}\nonumber \\
&=&e^{n-1}p+\sum_{\substack{s\mid n\\ s\neq 1, n}}e^{s-1}p^{\frac{n}{s}}\nonumber \\
&\leq &e^{n-1}p+\sum_{\substack{s\mid n\\ s\neq 1, n}}p^{s+\frac{n}{s}-1}\label{eq:pbigger}\\
&\leq &e^{n-1}p+np^{1+n/2}\label{eq:bound},
\end{eqnarray}
where \eqref{eq:pbigger} follows from our assumption that $e<p$, and \eqref{eq:bound} from the fact that $n$ has at most $n$ divisors, and for each divisor $s$ we have $s+\frac{n}{s}\leq 2+\frac{n}{2}$ (except $s=1$ or $s=n$, which we have excluded). Since
\[\frac{e^{n-1}p+np^{1+n/2}}{p^{n-1}}\to 0\] as $n\to \infty$, the inequality above bounds $n$. 
\end{proof}

\begin{remark}
In the case $e=2$, $p=3$, one can check that the inequality \eqref{eq:bound} implies $n\leq 7$, and so for each $n\geq 8$ the polynomial $f_n(c)\in \FF_3[c]$ has a primitive prime factor. On the other hand, one can confirm by computation that
 $f_n$ has a primitive prime factor for every $1\leq n\leq 7$, and so in fact for every $n\geq 1$ there is a value of $c$ such that $y^2=x^3+c$ has a critical path of length exactly $n$.
\end{remark}

\begin{remark}
For unicritical correspondences $y^e=x^d+c$, the height bound above becomes much simpler. If this correspondence is PCC, then $|c|_v\leq 1$ for any non-archmidean $v$. But, also, if $|x|^d>2\max\{1, |c|\}$ at an archimedean place, we have
\[|y|=|x^d+c|^{1/e}\geq \left(1-1/2\right)^{1/e}|x|^{d/e}>(1-1/2)^{1/e}2^{(d-e)/e}|x|>|x|.\]
By induction, we see that $0$ can be contained in a preperiodic path only if $|c|^{1/e}\leq 2^{1/d}\max\{1, |c|^{1/d}\}$, in other words, only if $|c|\leq 2^{e/(d-e)}$. This gives
\[h(c)\leq \left(\frac{e}{d-e}\right)\log 2.\]
\end{remark}

\end{document}